\documentclass[pdflatex,sn-mathphys-num]{sn-jnl}

\usepackage{graphicx}%
\usepackage{multirow}%
\usepackage{amsmath,amssymb,amsfonts}%
\usepackage{amsthm}%
\usepackage{mathrsfs}%
\usepackage[title]{appendix}%
\usepackage{xcolor}%
\usepackage{textcomp}%
\usepackage{manyfoot}%
\usepackage{booktabs}%
\usepackage{algorithm}%
\usepackage{algorithmicx}%
\usepackage{algpseudocode}%
\usepackage{listings}%
\usepackage{enumitem}
\theoremstyle{thmstyleone}%
\newtheorem{proposition}{Proposition}% 

\theoremstyle{thmstyletwo}%
\newtheorem{remark}{Remark}%

\newtheorem{experiment}{Test Case}[section]

\theoremstyle{thmstylethree}%

\newcommand{\lrp}[1]{\left( #1 \right)}
\newcommand{\lrc}[1]{\left\{ #1 \right\}}
\newcommand{\lrb}[1]{\left[ #1 \right]}
\newcommand{\lrv}[1]{\left| #1 \right|}

\newcommand{\norm}[1]{\left\| #1 \right\|}
\newcommand{\atob}[2]{#1, \ldots, #2}

\newcommand{\Dt}{{\Delta t}}

\newcommand{\K}{K}
\renewcommand{\L}{L}

\newcommand{\Kblock}[1]{\mathbf{#1}^{\lrp{\K}}}
\newcommand{\KblockDt}[1]{\mathbf{#1}^{\Dt, \lrp{\K}}}

\renewcommand{\j}{\ddot{\imath}}
\newcommand{\ii}{j}
\newcommand{\jj}{\ddot{\jmath}}

\newcommand{\T}{\mathbb{T}}
\newcommand{\U}{\mathcal{U}}
\newcommand{\V}{\mathcal{V}}

\newcommand{\Idx}{I}

\newcommand{\timeinterval}{\mathcal{I}}
\newcommand{\timeintervalpart}{\mathcal{I}}

\newcommand{\lengthtimeintervalpart}[1]{|\mathcal{I}_{#1}|}

\newcommand{\pardif}[2][\!]{\frac{\partial #1}{\partial #2}}

\newcommand{\tn}[1]{\textnormal{#1}}

\algrenewcommand{\algorithmiccomment}[1]{\# #1}

\begin{document}
	\thispagestyle{empty}
	\!
	
	\vspace{6cm}
	
	\begin{center}
		This preprint has not undergone peer review or any post-submission improvements or corrections. The Version of Record of this article is published in BIT Numerical Mathematics, and is available online at \url{https://doi.org/10.1007/s10543-026-01130-y}.
	\end{center}
	
	\setcounter{page}{0}
	
	\clearpage
	
	\title[Goal Oriented Splitting]{Splitting schemes for ODEs with goal--oriented error estimation}

	\author*[1]{\fnm{Erik} \sur{Weyl}}\email{erik.weyl@math.uni-wuppertal.de}
	%orcid 0009-0004-7111-7457
	
	\author[1]{\fnm{Andreas} \sur{Bartel}}\email{bartel@math.uni-wuppertal.de}
	%orcid 0000-0003-1979-179X
	\author[2]{\fnm{Manuel} \sur{Schaller}}\email{manuel.schaller@math.tu-chemnitz.de}
	%orcid 0000-0002-8081-5108
	
	\affil*[1]{\orgdiv{Department of Mathematics and Computer Science}, \orgname{Bergische Universit\"at Wuppertal}, \orgaddress{\street{Gau\ss stra\ss e 20}, \city{Wuppertal}, \postcode{D-42119}, %\state{North Rhine--Westphalia}, 
			\country{Germany}}}
	
	\affil[2]{\orgdiv{Faculty of Mathematics}, \orgname{TU Chemnitz}, \orgaddress{\street{ Reichenhainer Stra\ss e 41}, \city{Chemnitz}, \postcode{09126}, 
			%\state{State}, 
			\country{Germany}}}

	\abstract{We present a hybrid a--priori/a--posteriori goal oriented error estimator for a combination of dynamic iteration-based solution of ordinary differential equations discretized by finite elements. Our novel error estimator combines estimates from classical dynamic iteration methods, usually used to enable splitting--based distributed simulation, and from the dual weighted residual method to be able to evaluate and balance both, the dynamic iteration error and the discretization error in desired quantities of interest. The obtained error estimators are used to conduct refinements of the computational mesh and as a stopping criterion for the dynamic iteration. In particular, we allow for an adaptive and flexible discretization of the time domain, where variables can be discretized differently to match both goal and solution requirements, e.g.\ in view of multiple time scales. 
		We endow the scheme with efficient solvers from numerical linear algebra to ensure its applicability to complex problems.
		Numerical experiments compare the adaptive approach to a uniform refinement.
	}
	

	
	\maketitle
	
	%----------------------------
	\section{Introduction}
	\label{sec:introduction}	
	%----------------------------
	In many applications, ordinary differential equations (ODEs) are used to model the temporal evolution of dynamic processes. In numerous real--world phenomena, multiple subsystems of different nature interact; mathematically, this interaction is typically represented by coupling individual models. The resulting coupled systems often exhibit pronounced differences in time scales, degrees of nonlinearity, or dimensionality.
	
	Dynamic iteration, see e.g. \cite{lelarasmee1982waveform}, is an approximation method that allows to numerically simulate these models by solving the subsystems in a decoupled, iterative manner. This enables the combination of dedicated solvers and facilitates the use of different time grids for the subsystems that respect the individual scales of the phenomena.
	
	In this work, we derive error estimators for dynamic iteration methods coupled with finite element schemes that allow for quantification of the error in a user--specified Quantity of Interest (QoI). In fact, we seek for a number of dynamic iterations and an (as coarse as possible) time grid such that the numerical solution approximates the true solution up to a given accuracy when measured in the QoI. The error is thereby composed of two different sources, the dynamic iteration error, and the finite element discretization error. 
	
	As a first ingredient we apply adaptive finite elements (FE) with independent discretizations for all components of the state. In recent years, it has been shown that these methods are indeed equivalent to certain Runge--Kutta schemes \cite{munoz2019explicit, munoz2019variational}. To estimate the error of the solution in a given QoI, and to perform adaptive refinements of the time grid, we use the well--established dual weighted residual (DWR) method \cite{becker2001optimal}. Thereby, one solves an adjoint problem, which yields sensitivity factors for the impact of errors through the lens of the QoI. Since the seminal paper mentioned before, goal-oriented methods have been developed in various directions \cite{Becker2000,bruchhauser2022goal,Estep1995, meidner2007adaptive,Gruene22}. As a second component, we develop a goal oriented estimator for the error caused by splitting the problem based on dynamic iteration theory \cite{burrage1995parallel}. 
	
	We briefly provide an overview of relevant literature and relate it to our novelties.
	The combination of multirate in time and DWR based error estimation and refinement has for instance been applied to coupled flow and transport problems \cite{bruchhauser2022goal}. %In this paper, 
	Here, we employ the more general multiadaptive approach \cite{logg2003multi} in which no similarity for grids of different solution components is proscribed. Moreover, the localized error estimators naturally lead to variable steps sizes for the components of the solution. Controlling the iteration and discretization error has been applied, e.g., in~\cite{becker1995adaptive}, but from the perspective of using the iterative solvers for the discrete problem. Another approach to the goal oriented error estimation of discretization and iteration errors was presented in~\cite{rannacher2013adaptive}. The difference to our approach is that there the problem was first discretized and the discrete problem solved iteratively. The resulting error estimator was based entirely on DWR theory.   In~\cite{chaudhry2016posteriori} for instance, a two--stage approach was presented in which the error estimators obtained in a first coarse stage were used to design the fine mesh. The fine mesh construction is based on groups of time steps to exploit error cancellation. The application of the DWR method to the problems arising from dynamic iteration has been investigated before to obtain error estimators~\cite{estep2012posteriori, chaudhry2013posteriori, chaudhry2015posteriori}.
	
	In this work, by adopting a view of the dynamic iteration process as a single system of ODEs, we obtain a novel error estimator based on a holistic view of the method.
	
	This paper proceeds as follows: the splitting scheme and our interpretation in the light of the DWR method are introduced in Section~\ref{sec:unified}. In Section~\ref{sec:difem}, we describe the multirate discretization of the split problems derived before. Subsequently, Section~\ref{sec:errorest} deals with the implementation details and the balancing of error estimators for splitting and discretization. Numerical results are reported in Section~\ref{sec:experiments}. Finally, a conclusion and further possible research direction are given in Section~\ref{sec:conclusion}.

	%----------------------------
	\section{Global view on dynamic iteration} % as dynamical system
	\label{sec:unified}
	%----------------------------
	In this work, we consider the following problem formulation: let $m \in \mathbb{N}$ and time interval $\timeinterval := \lrb{t_0, t_n} \subset \mathbb{R}$ be given. We seek $U \in \mathcal{U} := H^1 \lrp{\timeinterval, \mathbb{R}^m}$ which satisfies the initial value problem
	\begin{equation}
		\begin{split}
			\dot{U} \lrp{t} + B U \lrp{t}	& = Y \lrp{t}, \quad t \in \timeinterval\\
			U \lrp{t_0}						& = U_{t_0},
		\end{split} \label{eq:1010}
	\end{equation}
	with given coefficient matrix $B \in \mathbb{R}^{m \times m}$, continuous right--hand side $Y : \timeinterval \rightarrow \mathbb{R}^m$ and initial value $U_{t_0} \in \mathbb{R}^m$.
	Furthermore, to streamline the presentation and to not hide the main novelties behind technical details, we consider a discrete quantity of interest~(QoI) 
	\begin{equation}
		J \lrp{U} := \sum_{r = 1}^R J_r \cdot U \lrp{\tau_r} \label{eq:1020}
	\end{equation}
	with $R\in \mathbb{N}$ time points of interest $\tau_r\in \mathcal{I}$ and given weight vectors $J_r\in\mathbb{R}^m$ for $r = 1, \ldots, R$. We assume that $\tau_R = t_n$ which can always be done w.l.o.g.\ by setting $J_R=0$. 
	
	Next, we introduce the splitting and dynamic iteration (Section~\ref{subsec:dynamic-iteration}) and view the whole process of the iterations together as a stacked system (Section~\ref{subsec:stacked-view}). This novel view aims at providing the means to estimate the error in the QoI $J$ and adaptively refine the discretization in order to achieve accurate results at low cost.

	%-------------------------
	\subsection{Construction of splitting schemes}
	\label{subsec:dynamic-iteration}
	%-------------------------
	For solving the full system~\eqref{eq:1010}, we use a dynamic iteration (DI) scheme to iteratively compute an approximation by solving subsystems. Dynamic iteration---also waveform relaxation---was first proposed in \cite{lelarasmee1982waveform}. It translates ideas of iterative methods for linear systems to systems of differential equations: some occurrences of the unknown $U_i$ in \eqref{eq:1010} are replaced with a known guess $\check{U}_i$ and then the resulting simplified problem is solved. One obtains an approximation $\hat{U}$ to $U$. By
	iteration, it is possible to construct a sequence, which convergences to the solution $U$, see~\cite{burrage1995parallel}.
	
	Given an initial waveform $U_0 \in \mathcal{U}$ satisfying $U_0 \lrp{t_0} = U_{t_0}$, e.g. $U_0 \equiv U_{t_0}$, a DI splitting scheme for the ODE initial value problem \eqref{eq:1010} may be constructed by choosing a splitting matrix $S \in \lrc{0, 1}^{m \times m}$ and defining
	\begin{align}
		\hat{B} := S \ast B \in \mathbb{R}^{m\times m}, & & \check{B} := B - \hat{B} \label{eq:2010}
	\end{align}
	where $\ast$ denotes elementwise multiplication. Then, for $k \geq 1$, $U_k$ is defined from $U_{k-1}$ as solution of the initial value problem
	\begin{equation}
		\begin{split}
			\dot{U}_k \lrp{t} + \hat{B} U_k \lrp{t} & = Y \lrp{t} - \check{B} U_{k - 1} \lrp{t}, \quad t \in \timeinterval\\
			U_k \lrp{t_0}	& = U_{t_0}.
		\end{split} \label{eq:2020}
	\end{equation}
	Notice, if the entry $s_{i, j}$ of $S$ is 0, the unknown $u_{k, j}$ is replaced by the previous iterate $u_{k - 1, j}$ in $i$--th equation of \eqref{eq:2020}. Therefore, $S = 1^{m \times m}$ corresponds to the original, fully coupled Setting \eqref{eq:1010}.		
	Prominent splitting schemes are:
	\begin{description}
		\item[Jacobi splitting.] Let $S = I$. In this case, during each solve, the equations are decoupled from one another making it possible to solve them all in parallel.
		\item[Gau\ss--Seidel splitting.] Let $S$ be lower triangular. This yields a sequential scheme where new results are immediately used for the computation of the next equation.
	\end{description}		
	Further splitting methods as block splitting, Picard iteration, underrelaxation and overlapping  can be also encoded via a splitting matrix $S$ and may be used in our proposed framework. In particular, overlapping can improve convergence rates~\cite{Jeltsch1995}.       
	
	Considering the split model~\eqref{eq:2020} instead of \eqref{eq:1010} results in a splitting error 
	\begin{equation}\label{eq:splitting-error}
		E_k \lrp{t} := U \lrp{t} - U_k \lrp{t} .
	\end{equation}
	The following result provides an estimate of the splitting error in the QoI and hence provides a goal oriented perspective. By $\norm{\cdot}$, we denote the Euclidean vector norm and the induced matrix norm.
	
	\begin{proposition}
		\label{prop:errDI}
		We consider problem~\eqref{eq:1010}. Given a splitting matrix $S$ with split model~\eqref{eq:2020}, then
		the error $E_{\K}$~\eqref{eq:splitting-error} after $\K$ iterations 
		%is bounded by
		% 
		leads to an error in the QoI \eqref{eq:1020} $J \lrp{E_{\K}}$, which is bounded as follows
		\begin{equation}
			\nu := \lrv{J \lrp{E_{\K}}} \leq \lrp{- \frac{L_2}{L_1}}^{\K} \sup_{t \in \timeinterval} \lrp{\norm{E^{\lrp{0}}}} \sum_{r = 1}^R \norm{J_r} \lrp{1 - e^{L_1 t_r} \sum_{k = 0}^{\K - 1} \frac{\lrp{- L_1 t_r}^k}{k !}} \label{eq:2030}
		\end{equation}
		with $L_1 = \mu_{\max} \lrp{\frac{\hat{B} + \hat{B}^T}{2}}$,
		$L_2 = \norm{\check{B}}$.
	\end{proposition}
	
	\begin{proof}
		Using the triangle and Cauchy--Schwarz inequalities, we get
		\begin{equation}
			\lrv{J \lrp{E_{\K}}} = \lrv{\sum_{r = 1}^R J_r \cdot E_{\K} \lrp{t_r}} \leq \sum_{r = 1}^R \lrv{J_r \cdot E_{\K} \lrp{t_r}} \leq \sum_{r = 1}^R \norm{J_r} \norm{E_{\K} \lrp{t_r}} . \label{eq:2032}
		\end{equation}
		It now remains to bound $\norm{E_{\K} \lrp{t_r}}$. Thereby, we rewrite \eqref{eq:2020} as
		\begin{equation*}
			\dot{U}_k \lrp{t} = F \lrp{U_k, U_{k - 1}, t} := Y \lrp{t} - \hat{B} U_k \lrp{t} - \check{B} U_{k - 1} \lrp{t}
		\end{equation*}
		and note that the right hand side $F$ satisfies the following Lipschitz conditions:
		\begin{align*}
			\lrp{F \lrp{U_a, \cdot, \cdot} - F \lrp{U_b, \cdot, \cdot}} \cdot \lrp{U_a - U_b}	& \leq L_1 \norm{U_a - U_b}^2\\
			\norm{F \lrp{\cdot, U_a, \cdot} - F \lrp{\cdot, U_b, \cdot}}						& \leq L_2 \norm{U_a - U_b}
		\end{align*}           
		The constants $L_1$ and $L_2$ in this case are the logarithmic norm of $\hat{B}$ and the matrix norm of $L_2$ respectively, as given above.
		
		Using \cite[Theorem 7.9.3]{burrage1995parallel}, we have for the error in the dynamic iteration %then get
		\begin{equation*}
			\sup_{t \in \lrb{t_0, t_r}} \lrp{\norm{E^{\lrp{k}} \lrp{t}}} \leq \lrp{- \frac{L_2}{L_1}}^k \lrp{1 - e^{L_1 t_r} \sum_{j = 0}^{k - 1} \frac{\lrp{- L_1 t_r}^j}{j !}} \sup_{t \in \lrb{t_0, t_r}} \lrp{\norm{E^{\lrp{0}}}} .
		\end{equation*}
		We can use this bound in \eqref{eq:2032} for each $t_r$. 
		Due to $\lrb{t_0, t_r}\subseteq \timeinterval$, we have
		%By using
		\begin{equation*}
			\sup_{t \in \lrb{t_0, t_r}} \lrp{\norm{E^{\lrp{k}} \lrp{t}}} \leq \sup_{t \in \timeinterval} \lrp{\norm{E^{\lrp{k}} \lrp{t}}},
		\end{equation*}
		which gives the stated error bound~\eqref{eq:2030}.
	\end{proof}
	
	\begin{remark}
		(i) In case of a dissipative system, i.e. $L_1 < 0$, we have 
		\begin{equation*}
			1 - e^{L_1 t_r} \sum_{k = 0}^{\K - 1} \frac{\lrp{- L_1 t_r}^k}{k !} = e^{L_1 t_r} \sum_{j = k}^{\infty} \frac{\lrp{- L_1 t_r}^j}{j !} \leq e^{L_1 t_r} \sum_{j = 0}^{\infty} \frac{\lrp{- L_1 t_r}^j}{j !} = 1 .
		\end{equation*}
		If we additionally have $- L_1 > L_2$, then~\eqref{eq:2030} implies that the dynamic iteration is contractive. More precisely, by choosing $S$ such that the coupling ratio $\lrv{\frac{L_2}{L_1}}$ is minimized, the speed of convergence can be improved.
		
		(ii) In the more general case where $L_1$ may be positive, it is still possible to obtain proof of asymptotic convergence \cite[Theorem 7.9.2]{burrage1995parallel}. However, in practice the error may blow up beyond the maximum machine number during early iterations~\cite{Bartel2023}.
	\end{remark}
	
	In practice, the number of iterations is limited. In this work, we aim to balance the splitting error estimator with the discretization error estimator introduced below. The details are given in Section~\ref{sec:errorest} (there, we also use the abbreviation $\nu$ for the splitting error in $J$).

	%---------------------
	\subsection{A holistic variational formulation for dynamic iteration} % methods
	\label{subsec:stacked-view}
	%---------------------
	
	For a total number of dynamic iteration steps $\K\in\mathbb{N}$, we stack the respective ODEs into a single system:
	\begin{equation}
		\lrp{\begin{array}{@{}c@{}} \dot{U}_1 \lrp{t}\\ \dot{U}_2 \lrp{t}\\ \vdots\\ \dot{U}_{\K - 1} \lrp{t} \\ \dot{U}_{\K} \lrp{t} \end{array}}
		+
		\lrb{\begin{array}{ccccc}
				\hat{B}		& 0			& \ldots	& 0			& 0\\
				\check{B}	& \hat{B}	& \ddots	& 			& 0\\
				0 		    & \ddots	& \ddots	& \ddots	& \vdots\\
				\vdots		& \ddots	& \ddots	& \hat{B}	& 0\\
				0			& \ldots	& 0	& \check{B}	& \hat{B}
		\end{array}}
		\lrp{\begin{array}{@{}c@{}} U_1 \lrp{t}\\ U_2 \lrp{t}\\ \vdots\\ U_{\K - 1} \lrp{t} \\ U_{\K} \lrp{t} \end{array}}
		= 
		\lrp{\begin{array}{c@{}l} Y \lrp{t} & - \check{B} U_0 \lrp{t}\\ Y \lrp{t} &\\  \vdots &\\ Y \lrp{t} &\\ Y \lrp{t} & \end{array}}
		\label{eq:2040}
	\end{equation}
	with initial waveform $U_0 \lrp{t}$.
	This allows for a more rigorous derivation of the DWR based error estimators later (Section~\ref{sec:DWR}). Additionally, it makes the flow of information --- forward through iterations for the primal problem and backward in inverse order for the dual --- more transparent. It also highlights the influence of the initial waveform. On the other hand, for the practical evaluation of the variational problem, it is better to return to the individual steps.

	We rewrite \eqref{eq:2040} more compactly as
	\begin{equation}
		\Kblock{\dot{U}} \lrp{t} + \Kblock{B} \Kblock{U} \lrp{t} = \Kblock{Y} \lrp{U_0; t}\label{eq:2042}
	\end{equation}
	with initial value
	\begin{equation*}
		\Kblock{U} \lrp{t_0} = \Kblock{U}_{t_0} := \lrp{U_{t_0}, \ldots, U_{t_0}}^\top  .
	\end{equation*}
	For the goal functional, we apply the QoI defined in \eqref{eq:1020} to the final result, discarding all previous ones:
	\begin{equation*}
		\Kblock{J} \lrp{\Kblock{U}} := J \lrp{U_{\K}} .
	\end{equation*}

	To obtain a weak formulation of \eqref{eq:2040}, we consider the test space $\mathcal{V} \subseteq L^{\infty} \lrp{\timeinterval, \mathbb{R}^m}$, use stacked solutions 
	$\Kblock{U} \in \Kblock{\mathcal{U}} := \mathcal{U} \times \ldots \times \mathcal{U}$ and stacked test functions $\Kblock{V} \in \Kblock{\mathcal{V}} := \mathcal{V} \times \ldots \times \mathcal{V}$ to define
	\begin{align*}
		\Kblock{f} \!\lrp{\Kblock{U}, \Kblock{V}}	
		& \!:= \!\!\int_\timeinterval\!\! \lrp{\Kblock{\dot{U}} \!\lrp{t}
			+ \Kblock{B} \Kblock{U} \!\lrp{t}} \cdot \Kblock{V} \lrp{t} \tn{d} t + \Kblock{U} \lrp{t_0} \cdot \Kblock{V} \lrp{t_0}
		\\
		\Kblock{g} \!\lrp{U_0; \Kblock{V}}				
		& \!:= \!\!\int_\timeinterval\!\! \Kblock{Y} \lrp{U_0; t} \cdot \Kblock{V} \lrp{t} \tn{d} t 
		+ \Kblock{U}_{t_0} \cdot \Kblock{V} \lrp{t_0} .
	\end{align*}
	With these definitions, we have:
	
	% - - - - - - - - - - - - - - - - - - - - - -   
	\textbf{Stacked dual variational problem.}
	% - - - - - - - - - - - - - - - - - - - - - -      
	Given the initial value problem $\eqref{eq:1010}$.
	Let an initial waveform $U_0 \in \mathcal{U}$ be given which satisfies the initial condition in \eqref{eq:1010}.
	Then, the stacked variational problem is to find $\Kblock{U} \in \Kblock{\mathcal{U}}$ such that for every test function $\Kblock{V} \in \Kblock{\mathcal{V}}$ holds %the equality
	\begin{equation}
		\Kblock{f} \lrp{\Kblock{U}, \Kblock{V}} = \Kblock{g} \lrp{U_0; \Kblock{V}} 
		.
		\label{eq:2050}
	\end{equation}
	This stacked view now provides the basis for the arguments in the following section.
	
	%-----------------------------------------------			
	\subsection{Dual weighted residual method}
	\label{sec:DWR}	
	%-----------------------------------------------	
	Using the variational formulation~\eqref{eq:2050}, we derive a goal oriented error estimator for the dynamic iteration method~\eqref{eq:2040} based on the dual weighted residual (DWR) method~\cite{becker2001optimal}. 
	
	We briefly provide the derivation, which is standard for goal oriented refinements~\cite{becker2001optimal}.  To this end, we reinterpret the variational problem as the trivial constrained optimization problem to minimize the quantity of interest $\Kblock{J} \lrp{\Kblock{U}}$ for $\Kblock{U} \in \Kblock{\mathcal{U}}$ subject to~\eqref{eq:2050}. The solution to this problem can be determined by finding a stationary point $\lrp{\Kblock{U}_s, \Kblock{Z}_s} \in \Kblock{\mathcal{U}} \times \Kblock{\mathcal{U}}$ of the Lagrangian function
	\begin{equation*}
		L \lrp{\Kblock{U}, \Kblock{Z}} = \Kblock{J} \lrp{\Kblock{U}} + \Kblock{g} \lrp{\Kblock{Z}} - \Kblock{f} \lrp{\Kblock{U}, \Kblock{Z}} .
	\end{equation*}
	Consequently, first--order optimality condition is given by the two equations 
	\begin{align}
		\left.
		\pardif[L]{\Kblock{Z}} \lrp{\Kblock{U}, \Kblock{Z}} 
		\right|_{\Kblock{Z}=\Kblock{V}}
		& = \Kblock{g} \lrp{\Kblock{V}} - \Kblock{f} \lrp{\Kblock{U}, \Kblock{V}} = 0, 
		\label{eq:2060}\\
		\left.
		\pardif[L]{\Kblock{U}} \lrp{\Kblock{U}, \Kblock{Z}} 
		\right|_{\Kblock{U}=\Kblock{W}}
		& = \Kblock{J} \lrp{\Kblock{W}} - \Kblock{f} \lrp{\Kblock{W}, \Kblock{Z}} = 0  \label{eq:2062}
	\end{align}
	holding for all $\Kblock{V} \in \Kblock{\mathcal{U}}$ and $\Kblock{W} \in \Kblock{\mathcal{U}}$. We note that \eqref{eq:2060} is the (primal) variational equation \eqref{eq:2050} while \eqref{eq:2062} is the adjoint (or dual) equation for the adjoint (or dual) solution $\Kblock{Z}$.
	\vspace*{5pt}
	
	% - - - - - - - - - - - - - - - - - - - - - -      
	% \noindent
	\textbf{Stacked dual variational problem.}
	% \label{prob:adjoint}
	% - - - - - - - - - - - - - - - - - - - - - -      
	Find $\Kblock{Z} \in \Kblock{\mathcal{U}}$ such that for all $\Kblock{W} \in \Kblock{\mathcal{U}}$ holds
	\begin{equation}
		\Kblock{f} \lrp{\Kblock{W}, \Kblock{Z}} = \Kblock{J} \lrp{\Kblock{W}} \label{eq:2110}
	\end{equation}
	with
	\begin{equation*}
		\Kblock{f} \!\lrp{\Kblock{W}\!, \Kblock{Z}}	
		\!:=\!\! \int_\timeinterval \!\!\lrp{\Kblock{\dot{W}} \lrp{t} 
			+ \Kblock{B} \Kblock{W} \lrp{t}} \cdot \Kblock{Z}\! \lrp{t} \tn{d} t 
		+ \Kblock{W} \lrp{t_0} \cdot \Kblock{Z} \lrp{t_0} .
	\end{equation*}		
	Using integration by parts, we have for $\Kblock{f} \lrp{\Kblock{W}, \Kblock{Z}}$
	\begin{equation*}
		= \int_\timeinterval \lrp{- \Kblock{\dot{Z}} \lrp{t} 
			+ {\Kblock{B}}^\top \Kblock{Z} \lrp{t}} \cdot \Kblock{W} \lrp{t} \tn{d} t 
		+ \Kblock{W} \lrp{t_n} \cdot \Kblock{Z} \lrp{t_n} .
	\end{equation*}
	This shows that the adjoint problem has a terminal condition given by the weight $J_R$ at the end of $\timeinterval$: specifically $Z \lrp{t_n} = J_R$. Furthermore since the matrix $\Kblock{B}$ is transposed, the dynamic iteration steps are solved in reverse order.	
	
	In practice, we solve discrete versions of \eqref{eq:2060} and \eqref{eq:2062} with solutions in a finite--dimensional subspace. For these approximations, we have the following goal oriented error bound \cite[Proposition 2.2]{becker2001optimal}:
	
	\begin{proposition}
		\label{prop:DWR}
		Consider the weakly formulated primal and dual problems \eqref{eq:2060} and \eqref{eq:2062} and their respective analytic solutions $\Kblock{U}$ and $\Kblock{Z}$. Let $\KblockDt{U} \in \KblockDt{\U}$ be the solution to a discretized version of the primal problem. Then, the error in the quantity of interest,
		\begin{equation}
			\Kblock{J} \lrp{\KblockDt{E}} := \Kblock{J} \lrp{\Kblock{U}} - \Kblock{J} \lrp{\KblockDt{U}} \le \mu^{\Dt}
			\label{eq:2090}
		\end{equation}
		is bounded from above for any $\KblockDt{V} \in  \KblockDt{\mathcal{U}}$ with
		\begin{equation}
			\mu^{\Dt} \!:=\!\! \int_\timeinterval \!\!\lrp{\Kblock{Y} \lrp{t} \!-\! \Kblock{A} \KblockDt{\dot{U}} \lrp{t} \!-\! \Kblock{B} \KblockDt{U} \lrp{t}} \!\cdot\! \lrp{\Kblock{Z} \lrp{t} \!-\! \KblockDt{V} \lrp{t}} \tn{d} t .  \label{eq:2070}                
		\end{equation}
	\end{proposition}
	
	\begin{proof}
		The proposition is proved by \cite[Propositions 2.2 and 2.3]{becker2001optimal}, noting that both the left hand side $\Kblock{f} \lrp{\cdot, \cdot}$ and the right hand side $\Kblock{g} \lrp{\cdot}$ are linear. 
		In Appendix~\ref{sec:appendix050}, the proof specifically for our case is presented.
	\end{proof}
	
	To make the error estimator in \eqref{eq:2070} sharp, $\KblockDt{V}$ has to be a good approximation to $\Kblock{Z}$. In practice, we use the discrete adjoint result $\KblockDt{Z}$ as $\KblockDt{V}$. Since the exact adjoint $\Kblock{Z}$ is generally unknown, we approximate the adjoint error by the difference between two different interpolations (e.g. linear and quadratic) of $\KblockDt{Z}$.
	
	%-----------------------------
	\subsection{Decoupling}
	%-----------------------------
	\label{sec:decoupling}
	The presented \emph{holistic} formulation~\eqref{eq:2050} of variational problem~\eqref{eq:2040} is useful for the derivation of the DWR method. Of course, in practice, the iteration steps are usually solved sequentially. More precisely, we may use the block--triangular structure % can use that
	\begin{align*}
		\Kblock{B} = \lrb{\begin{array}{ccccc}
				\hat{B}		& 0			& \ldots	& 0			& 0\\
				\check{B}	& \hat{B}	& \ddots	& 0			& 0\\
				\vdots		& \ddots	& \ddots	& \ddots	& \vdots\\
				0			& 0			& \ddots	& \hat{B}	& 0\\
				0			& 0			& \ldots	& \check{B}	& \hat{B}
		\end{array}},
		& &
		{\Kblock{B}}^\top = \lrb{\begin{array}{ccccc}
				\hat{B}^\top	& \check{B}^\top	& \ldots	& 0			& 0\\
				0			& \hat{B}^\top		& \ddots	& 0			& 0\\
				\vdots		& \ddots		& \ddots	& \ddots	& \vdots\\
				0			& 0				& \ddots	& \hat{B}^\top	& \check{B}^\top\\
				0			& 0				& \ldots	& 0			& \hat{B}^\top
		\end{array}}
	\end{align*}
	to decouple the unified variational formulations: We define for $U \in \mathcal{U}$ and $V \in \mathcal{V}$:
	\begin{align*}
		\hat{f} \lrp{U, V}		& := \int\timeinterval \lrp{\dot{U} \lrp{t} + \hat{B} U \lrp{t}} \cdot V \lrp{t} \tn{d} t + U \lrp{t_0} \cdot V \lrp{t_0}
		\\
		&  = \int_\timeinterval \lrp{- \dot{V} \lrp{t} + \hat{B}^\top V \lrp{t}} \cdot U \lrp{t} \tn{d} t + V \lrp{t_n} \cdot U \lrp{t_n},
		\\
		\check{f} \lrp{U, V}	& := \int_\timeinterval \check{B} U \lrp{t} \cdot V \lrp{t} \tn{d} t
		%\\
		%& 
		\quad= \quad \int_\timeinterval \check{B}^\top V \lrp{t} \cdot\! U \lrp{t} \tn{d} t,
		\\
		g \lrp{V}				& := \int_\timeinterval Y \lrp{t} \cdot V \lrp{t} \tn{d} t + U_{t_0} \!\cdot V \lrp{t_0},
		%\\
		\quad\;
		h \lrp{U}				%& 
		:= \int_\timeinterval J_1^\top \lrp{t} \cdot U \lrp{t} + J_2^\top \!\cdot\! U \lrp{t_n} .
	\end{align*}
	Then, the primal problem reads: find $U_k \in \U$ such that for all $V \in \V$ holds
	\begin{equation}
		\hat{f} \lrp{U_k, V} = g \lrp{V} - \check{f} (U_{k-1}, V) 
		\quad (k = \atob{1}{\K}). \label{eq:2210}
	\end{equation}
	
	The dual problem proceeds in reverse order, $k = \atob{\K}{1}$, and backward in time; this coincides with %the results in 
	\cite{estep2012posteriori}. Thus, we seek $Z_k \in \V$ such that for all $W \in \U$ holds
	\begin{equation}
		\hat{f} \lrp{Z_k, W} = 
		\left\{ \begin{array}{cc} 
			h \lrp{W} & \tn{ for } k = \K,
			\\
			- \check{f} \lrp{W, Z_{k + 1}} & \tn{ otherwise}. 
		\end{array} \right.  
		\label{eq:2220}
	\end{equation}

	%-------------------------        
	\section{Discretization}
	\label{sec:difem}		
	%-------------------------        
	We discuss the discretization of the variational problems~\eqref{eq:2050} and \eqref{eq:2110} via finite dimensional trial and test spaces $\U^{\Dt}$ and $\V^{\Dt}$.

	%--------------------------------------------------------	
	\subsection{Multiadaptive discretization schemes}
	\label{sec:multigrid}				
	%--------------------------------------------------------	
	Many naturally occurring problems exhibit multirate behavior, where the components operate on time scales, which differ drastically. This occurs frequently in electrical circuit simulation where %latency concentrates 
	the activity in concentrated on a subset of components~\cite{bartel2002multirate}. Another example is in astrophysics where Keplerian motion results in significant velocity differences between orbiting bodies \cite{engstler1997multirate}. To efficiently address this, different discretizations tailored to the frequencies of the respective components should be used.
	
	The term \emph{multiadaptive} originates from \cite{logg2003multi, logg2004multi}, where it is used to describe methods that allowed each component to have its own time grid, order of method and quadrature. In this work, to concisely illustrate the core of our idea, we restrict ourselves to the first of these three features, while including the other features is straightforward and left for future research. 
	
	For each component $i = \atob{1}{m}$ of the state solving the ODE~\eqref{eq:1010}, we partition the simulation time interval $\timeinterval = [t_0,t_n]$ into $n_{i}$ cells parametrized by the nodes
	\begin{equation*}
		t_0 = t_{i, 0} < t_{i, 1} < \ldots < t_{i, n_{i} - 1} < t_{i, n_{i}} = t_n .
	\end{equation*}
	Let $N := \sum_{i = 1}^m n_{i}$ be the total number of cells. For $\ii = \atob{1}{n_{i}}$, we define
	\[ 
	\timeintervalpart_{i, \ii} := \left[t_{i, \ii - 1}, t_{i, \ii}\right[ \quad \text{with} \quad 
	\lengthtimeintervalpart{i, \ii} := t_{i, \ii} - t_{i, \ii - 1}.
	\]  
	And for notational ease, we use:
	\[
	\timeintervalpart_{i,0} = \left] -\infty, t_{t_{i,0}} \right[, 
	\qquad 
	\timeintervalpart_{i,n_{i}+1} = \left[ t_{i,n_{i}}, \infty \right[.
	\]
	Furthermore, we define the sets of admissible indices
	\begin{align*}
		\Idx	& := \lrc{\lrp{i, \ii} \in \mathbb{Z}^2 \; | \; 1 \leq i \leq m, \; 1 \leq \ii \leq n_{i}}, 
		%\\
		\quad
		\Idx_0 %&
		:= \Idx \cup \lrc{\lrp{i, 0} \in \mathbb{Z}^2 \; | \; 1 \leq i \leq m}.
	\end{align*}		
	For the sake of ease and clarity, we will use the same discretization for all dynamic iteration steps ($k$) in one sweep. 
	This has several practical benefits: %in particular, 
	as shown later in Proposition~\ref{prop:coefficients}, it allows the reuse of computed discrete systems across dynamic iterations; secondly, it also allows to compute error estimators for variable numbers of iteration steps. Consequently, we can dynamically adjust the number of dynamic iteration steps before grid refinements. In principle, our framework does allow using different grids in each step if this is needed.
	
	Now, we introduce three options for the finite dimensional space of the individual components $i = \atob{1}{m}$ of the state:
	\begin{enumerate}
		\item[(a)] piecewise constant functions:
		\begin{equation*}
			\U_{a, i}^{\Dt} := \Bigl\{ u \; \Bigl| \;\;  
			\forall \ii = \atob{1}{n_{i}+1} : \; u\bigl|_{\timeintervalpart_{i, \ii}} 
			\equiv \tn{const} 
			\}
		\end{equation*}
		with respective indicator functions $\phi_{i, \ii}$ as basis:
		\begin{equation*} \begin{array}{r@{}l@{\quad}ll}
				\phi_{i, 0} \lrp{t} & := 
				\left\{ \begin{array}{cc} 1 & t \in \left] - \infty, t_{i, 1} \right[\\ 
					0 & \tn{ otherwise} 
				\end{array} 
				\right. 
				& \dot{\phi}_{i, 0} \lrp{t} \!= - \delta \lrp{t - t_{i, 1}} & 
				\\[0.25ex]
				\phi_{i, \ii} \lrp{t} & := 
				\left\{ \begin{array}{cc} 1 & t \in \timeintervalpart_{i,\ii}
					\\ 
					0 & \tn{ otherwise} 
				\end{array} \right. 
				& \dot{\phi}_{i, 0} \!\lrp{t} \!=\! \delta \lrp{t \!-\! t_{i, \ii - 1}} - \delta \lrp{t \!-\! t_{i, \ii}}, & \ii \!=\! \atob{1}{n_{i} \!-\! 1}
				\\[0.25ex]
				\phi_{i, n_{i}} \!\lrp{t} 
				&:= \left\{ \begin{array}{cc} 1 & t \in \left[ t_{i, n_{i}}, \infty \right[\\ 0 & \tn{ otherwise} 
				\end{array} \right. 
				& \dot{\phi}_{i, 0} \lrp{t} = \delta \lrp{t - t_{i, n_{i}}}; & 
			\end{array} 
		\end{equation*}
		
		\item[(b)] as a slight modification of (a), we have $\U_{b, i}^{\Dt} $ defined as in (a) just the cells are minutely modified $\tilde{\timeintervalpart}_{i,\ii}:= \left] t_{i,\ii-1}, t_{i,\ii}\right]$; this could be considered a sort of dual counterpart to $\U_{a, i}^{\Dt}$, where we start at $t_n$ and go backward in time from there.
		
		\item[(c)] piecewise linear and globally continuous functions
		\begin{equation*}
			\U_{c, i}^{\Dt}	
			:= \lrc{u \in \mathcal{C}^0 \lrp{\timeinterval, \mathbb{R}} \;\Bigl|\; 
				\forall \ii = \atob{1}{n_{i}}: u\bigl|_{\timeintervalpart_{i, \ii}} \in \mathcal{P}_1 \lrp{\timeintervalpart_{i, \ii}}%
			},
		\end{equation*}
		where $\mathcal{P}_1 \lrp{R}$ are polynomial of maximal degree one on the domain $R \subseteq \mathbb{R}$; suitable basis functions are the hat functions (for $\ii = \atob{0}{n_{i}}$):
		\begin{equation*}
			\phi_{i, \ii} \lrp{t} := \left\{ 
			\begin{array}{cl} 
				\frac{t - t_{i, \ii - 1}}{\Dt_{i, \ii - 1}} & t \in \timeintervalpart_{i, \ii - 1} \;\tn{ and }\; \ii \geq 1,
				\\[0.25ex] 
				\frac{t_{i, \ii + 1} - t}{\Dt_{i, \ii}} & t \in \timeintervalpart_{i, \ii} \;\tn{ and }\; \ii \leq n_{i} - 1,
				\\[0.25ex]
				0 & \tn{otherwise}. \end{array} \right.  
		\end{equation*}
	\end{enumerate}
	For $x \in \lrc{a, b, c}$, the overall spaces (for all components) are then given by
	\begin{equation*}
		\U_x^{\Dt} := \lrc{U = \lrp{u_1, \ldots, u_m} \;|\; \forall i = \atob{1}{m} : u_{i} \in \U_{x, i}^{\Dt}}
	\end{equation*}
	with the basis functions $\Phi_{i, \ii} \lrp{t} := e_i \phi_{i, \ii} \lrp{t}$ for $\lrp{i, \ii} \in \Idx_0$ and canonical basis vector $e_{i} \in\mathbb{R}^m$.
	Different FE schemes now arise by choosing $\U^{\Dt}$ and $\V^{\Dt}$ from $\U_a^{\Dt}$, $\U_b^{\Dt}$ and $\U_c^{\Dt}$. Of particular interest to us are the choices $\U^{\Dt}:= \U_a^{\Dt}$ and $\V^{\Dt}:= \U_b^{\Dt}$ yields a time--stepping scheme equivalent to the explicit Euler method as well as $\U^{\Dt}:= \U_c^{\Dt}$ and $\V^{\Dt}:= \U_b^{\Dt}$ for the Crank--Nicolson scheme. Of course, other options are also possible.

	%----------------------------------------------------    
	\subsection{Solution of the discretized problem}
	%----------------------------------------------------  
	Here, we consider one fixed step $k$ ($\in $\{\atob{1}{\K}\}$)$ of dynamic iteration step.
	We seek an approximation $U_k^{\Dt}$ to the analytic solution $U_k$ of \eqref{eq:2210}:
	\begin{equation}
		U_k \lrp{t} \approx 
		U_k^{\Dt} \lrp{t}
		:= \sum_{i = 1}^m \sum_{\ii = 0}^{n_{i}} u_{k, i, \ii}^{\Dt} \cdot \Phi_{i, \ii} \lrp{t} . \label{eq:3210}
	\end{equation}
	Likewise, the approximation of the solution of the adjoint equation~\eqref{eq:2110} is given by
	\begin{equation}
		Z_k \lrp{t} \approx 
		Z_kt^{\Dt}  \lrp{t} 
		:= \sum_{i = 1}^m \sum_{\ii = 0}^{n_{i}} z_{k, i, \ii}^{\Dt} \cdot \Phi_{i, \ii} \lrp{t} . \label{eq:3220}
	\end{equation}
	The unknown coefficients $u_{k, i, \ii}^{\Dt}$ and $z_{k, i, \ii}^{\Dt}$ for $\lrp{i, \ii} \in \Idx_0$ are obtained by solving by solving the finite--dimensional discrete primal and dual variational problems.
	\vspace{5pt}
	
	%- - - - - - - - - - - - - - - - - - - - - 
	\textbf{Discrete primal formulation.}
	%- - - - - - - - - - - - - - - - - - - - - 
	The discrete primal variational problem associated with the ODE~\eqref{eq:2210} is to find coefficients $u_{k, i, \ii}^{\Dt}$ for $\lrp{i, \ii} \in \Idx_0$ such that for all $\Phi_{\j, \jj} \in \V^{\Dt}$ holds
	\begin{equation} \label{prob:discretePrimal}
		\hat{f} \lrp{U_k^{\Dt}, \Phi_{\j, \jj}} = g \lrp{\Phi_{\j, \jj}} - \check{f} \lrp{U_{k-1}^{\Dt}, V^{\Dt}}. %\label{eq:3230}
	\end{equation}
	Using the ansatz for $U_k^{\Dt}$ \eqref{eq:3210} and linearity, we can write the linear system \eqref{prob:discretePrimal} as:
	\begin{align*}
		& \underbrace{\lrb{\begin{array}{ccc}
					\hat{f} \lrp{\Phi_{1, 0},  \Phi_{1, 0}} & \ldots & \hat{f} \lrp{\Phi_{m, n_m},  \Phi_{1, 0}}\\
					\vdots & & \vdots\\
					\hat{f} \lrp{\Phi_{1, 0},  \Phi_{m, n_m}} & \ldots & \hat{f} \lrp{\Phi_{m, n_m},  \Phi_{m, n_m}}
		\end{array}}}_{\hat{F}}
		\lrp{\begin{array}{c} u_{k, 1, 0}^{\Dt}\\ \vdots\\ u_{k, m, n_m}^{\Dt} \end{array}} \nonumber\\
		= \underbrace{\lrp{\begin{array}{c} g \lrp{\Phi_{1, 0}}\\ \vdots\\ g \lrp{\Phi_{m, n_m}} \end{array}}}_{G} - & \underbrace{\lrb{\begin{array}{ccc}
					\check{f} \lrp{\Phi_{1, 0},  \Phi_{1, 0}} & \ldots & \check{f} \lrp{\Phi_{m, n_m},  \Phi_{1, 0}}\\
					\vdots & & \vdots\\
					\check{f} \lrp{\Phi_{1, 0},  \Phi_{m, n_m}} & \ldots & \check{f} \lrp{\Phi_{m, n_m},  \Phi_{m, n_m}}
		\end{array}}}_{\check{F}} \lrp{\begin{array}{c} u_{k - 1, 1, 0}^{\Dt}\\ \vdots\\ u_{k - 1, m, n_m}^{\Dt} \end{array}} . %\label{eq:2165}
	\end{align*}
	
	%- - - - - - - - - - - - - - - - - - - - -        
	\textbf{Discrete dual formulation.}
	\label{prob:discreteDual}
	%- - - - - - - - - - - - - - - - - - - - -        
	The discrete adjoint variational problem corresponding to the adjoint problem~\eqref{eq:2110} is to find coefficients $z_{k, i, \ii}^{\Dt}$ for $\lrp{i, \ii} \in \Idx_0$ such that for all $\Phi_{i, \ii} \in \U^{\Dt}$ holds
	\begin{equation}
		\hat{f} \lrp{\Phi_{i, \ii}, Z_k^{\Dt}} = \left\{ 
		\begin{array}{cc} 
			h \lrp{\Phi_{i, \ii}} & \tn{ if } k = \K,
			\\ 
			- \check{f} \lrp{\Phi_{i, \ii}, Z_{k + 1}^{\Dt}} & \tn{ otherwise.}
		\end{array} \right. \label{eq:3250}
	\end{equation}
	
	By definition of $Z_k^{\Dt}$ in \eqref{eq:3220} and linearity, we obtain the system of equations for the discrete adjoint problem. for $k = \K$, we get:
	\begin{equation*}
		\underbrace{\lrb{\begin{array}{ccc}
					\hat{f} \lrp{\Phi_{1, 0},  \Phi_{1, 0}} & \ldots & \hat{f} \lrp{\Phi_{1, 0},  \Phi_{m, n_m}}\\
					\vdots & & \vdots\\
					\hat{f} \lrp{\Phi_{n, n_m},  \Phi_{1, 0}} & \ldots & \hat{f} \lrp{\Phi_{m, n_m},  \Phi_{m, n_m}}
		\end{array}}}_{\hat{F}^{\top}}
		\lrp{\begin{array}{c} z_{k, 1, 0}^{\Dt}\\ \vdots\\ z_{k, m, n_m}^{\Dt} \end{array}}
		= 
		\underbrace{\lrp{\begin{array}{c} h \lrp{\Phi_{1, 0}}\\ \vdots\\ h \lrp{\Phi_{m, n_m}} \end{array}}}_{H}
	\end{equation*}
	and for $k = \atob{\K - 1}{1}$:
	\begin{align*}
		& \underbrace{\lrb{\begin{array}{ccc}
					\hat{f} \lrp{\Phi_{1, 0},  \Phi_{1, 0}} & \ldots & \hat{f} \lrp{\Phi_{1, 0},  \Phi_{m, n_m}}\\
					\vdots & & \vdots\\
					\hat{f} \lrp{\Phi_{n, n_m},  \Phi_{1, 0}} & \ldots & \hat{f} \lrp{\Phi_{m, n_m},  \Phi_{m, n_m}}
		\end{array}}}_{\hat{F}^\top}
		\lrp{\begin{array}{c} z_{k, 1, 0}^{\Dt}\\ \vdots\\ z_{k, m, n_m}^{\Dt} \end{array}} \nonumber\\
		= -
		& \underbrace{\lrb{\begin{array}{ccc}
					\check{f} \lrp{\Phi_{1, 0},  \Phi_{1, 0}} & \ldots & \check{f} \lrp{\Phi_{1, 0},  \Phi_{m, n_m}}\\
					\vdots & & \vdots\\
					\check{f} \lrp{\Phi_{n, n_m},  \Phi_{1, 0}} & \ldots & \check{f} \lrp{\Phi_{m, n_m},  \Phi_{m, n_m}}
		\end{array}}}_{\check{F}^{\top}}
		\lrp{\begin{array}{c} z_{k + 1, 1, 0}^{\Dt}\\ \vdots\\ z_{k + 1, m, n_m}^{\Dt} \end{array}} .
	\end{align*}
	
	\begin{proposition}
		\label{prop:coefficients}
		Given $\lrp{i, \ii}, \lrp{\j, \jj} \in \Idx_0$,
		the entries of $\hat{F}$, $\check{F}$, $G$ and $H$ read: 
		\begin{align*}
			\hat{f} \lrp{\Phi_{\j, \jj},  \Phi_{i, \ii}}		
			& = \int_\timeinterval \lrp{\delta_{\j, i} \dot{\phi}_{\j, \jj} \lrp{t} + \hat{b}_{i \j} \phi_{\j, \jj} \lrp{t}} \phi_{i, \ii} \lrp{t} \tn{d} t + \delta_{\j, i} \delta_{\jj, 0} \delta_{\ii, 0},
			\\
			\check{f} \lrp{\Phi_{\j, \jj},  \Phi_{i, \ii}}	
			& = \int_\timeinterval \lrp{\delta_{\j, i} \dot{\phi}_{\j, \jj} \lrp{t} + \check{b}_{i \j} \phi_{\j, \jj} \lrp{t}} \phi_{i, \ii} \lrp{t} \tn{d} t + \delta_{\j, i} \delta_{\jj, 0} \delta_{\ii, 0},
			\\
			g \lrp{\Phi_{i, \ii}}							
			& = \int_\timeinterval y_{i} \lrp{t} \phi_{i, \ii} \lrp{t} \tn{d} t + \delta_{\ii, 0} u_{t_0, i},
			%\\
			\qquad 
			h \lrp{\Phi_{\j, \jj}}							
			%& 
			= \sum_{r = 1}^R j_{t, \j} \phi_{\j, \jj} \lrp{t_r} .
		\end{align*}
	\end{proposition}
	\begin{proof}
		This follows from inserting the appropriate $\Phi_{i, \ii} = e_i \phi_{i, \ii}$ into the functions $\hat{f}$, $\check{f}$, $g$ and $h$ as defined in Section \ref{sec:decoupling}.
	\end{proof}
	
	\begin{remark} %\
		Note that $\hat{F}$, $\check{F}$, $G$ and $H$ do not depend on the dynamic iteration index $k$, hence they need to be computed only once in terms of the basis functions. The right hand sides $g$ an $h$ include the more general functions $Y$ and $J_1$ respectively and so must be approximated using quadrature. In this case, it is advantageous to use an approximation of the integral that is consistent with the left hand side.
	\end{remark}
	
	%--------------------------------------------            
	\section{Goal oriented mesh refinement}
	\label{sec:errorest}	
	%--------------------------------------------            
	We describe the computation of the error estimators and the resulting mesh refinement for the DI method. First, we provide an overview of the proposed error estimation and refinement method. 
	To this end, we refer to the meshes described in Section~\ref{sec:multigrid} as $\T^l$ with the index $l$ denoting the refinement level. Thus, from the initially given mesh $\T^0$, we obtain iteratively finer meshes using error estimators from the previous level.
	
	We estimate the error as follows:
	\begin{equation*}
		\lrv{J \lrp{U} - J \lrp{U_{\K_l}^{\Dt}}} \leq \lrv{J \lrp{U} - J \lrp{U_{\K_l}}} + \lrv{J_{\K_l} \lrp{U} - J \lrp{U_{\K_l}^{\Dt}}} \leq \nu_l + \mu_{\K_l}^{\Dt_l} .
	\end{equation*}
	Here, $\nu_l$ and $\mu^{\Dt_l}$ are estimates of the splitting and discretization errors, respectively. We already have a bound for $\nu_l$ in %the form of 
	Proposition~\ref{prop:errDI}. The basis for evaluating $\mu^{\Dt_l}$ is provided in Proposition~\ref{prop:DWR}; and in the following, we discuss how to obtain the localized error estimators %from this 
	efficiently.
	
	%---------------------------------------------------------------------
	\subsection{Local goal oriented discretization error estimators}
	\label{sec:localDiscreteGO}
	%---------------------------------------------------------------------
	
	The upper bound of the total error $\Kblock{J} \lrp{\KblockDt{E}} = J \lrp{E_k^{\Dt}}$ is given by \eqref{eq:2090} and can be split as follows:
	\begin{align}
		& \int_\timeinterval \lrp{\Kblock{Y} \lrp{t} - \KblockDt{\dot{U}} \lrp{t} - \Kblock{B} \KblockDt{U} \lrp{t}} \cdot \lrp{\Kblock{Z} \lrp{t} - \KblockDt{Z} \lrp{t}} dt 
		\label{eq:3510}
		\\
		=	& \sum_{k = 1}^{\K} \sum_{i = 1}^m \sum_{\ii = 1}^{n_{i}} \int_{\timeintervalpart_{i, \ii}} \underbrace{\lrp{y_{i} \lrp{t} - \dot{u}_{k, i}^{\Dt} \lrp{t} - \hat{B}_{i \cdot} \cdot U_k^{\Dt} \lrp{t} - \check{B}_{i \cdot} \cdot U_{k - 1}^{\Dt} \lrp{t}}}_{\rho_i \lrp{U_k^{\Dt}, U_{k - 1}^{\Dt}, t}} \underbrace{\lrp{z_{i} \lrp{t} - z_{k, i}^{\Dt} \lrp{t}}}_{e_{k, i}^* \lrp{t}} dt . \nonumber
	\end{align}
	This partition would allow to compute error estimators for every time step $\ii$ in every component $i$ in every DI step $k$. However, since we want to use the same grid for all $k$, we add up the errors over the iteration index $k$ for each $\lrp{i, \ii} \in \Idx$. If a cell is then refined, that refinement effects all iteration steps $k$.			
	\begin{equation}
		\mu_{i, \ii}^{\Dt} 
		:= \sum_{k = 1}^{\K} \lrv{\int_{\timeintervalpart_{i, \ii}} \rho_i \lrp{U_k^{\Dt}, U_{k - 1}^{\Dt}, t} e_{k, i}^* \lrp{t} dt} .\label{eq:3520}
	\end{equation}
	The use of absolute values is to prevent cancellation effects between the terms for different $k$.
	Then, a numerical estimator of the total $J$ error is given by
	\begin{equation*}
		\mu^{\Dt} := \sum_{i = 1}^m \sum_{\ii = 1}^{n_{i}} \mu_{i, \ii}^{\Dt} .
	\end{equation*}
	
	\begin{remark}
		%				\
		\begin{enumerate}[label = (\roman*)]	
			\item Note that due to the approximate nature of the computations there is no guarantee that it is actually an upper bound for the error in $J$. However, it can still be useful to guide the refinement.
			\item In \eqref{eq:3520}, we use the unknown exact solution $z_{i} \lrp{t}$ of the adjoint problem which is generally unknown. Therefore, % As such, 
			we use different approximations of $z_{i}$ in place of the term $\lrp{z_{i} \lrp{t} - z_{k, i}^{\Dt} \lrp{t}}$ (see below and cf. \cite[Section 5.1]{becker2001optimal}). %\hfill $\Box$
		\end{enumerate}
	\end{remark}
	
	%- - - - - - - - - - - - - - - - - -- - -- - - - - - - - - - - - - - - - 
	\noindent
	\textbf{Efficient computation of the adjoint and error estimators.}
	%- - - - - - - - - - - - - - - - - -- - -- - - - - - - - - - - - - - - - 
	In theory, adding another DI step changes the stacked problem \eqref{eq:2050} and thus would require evaluating the new adjoint problems for all $\tilde{k} = \atob{k}{1}$. However, this can be circumvented by exploiting the structure of the discrete adjoint problem \ref{prob:discreteDual}. Suppose we have made $\K$ steps in the process and have solved the discrete adjoint problems
	\begin{align*}
		\hat{F}^T Z_{\K}^{\Dt}  & = H,\\
		\hat{F}^T Z_{k}^{\Dt}   & = - \check{F}^\top Z_{k + 1}^{\Dt}, \qquad k = \atob{\K - 1}{1} .
	\end{align*}
	For the next iteration step, we would have to solve
	\begin{align*}
		\hat{F}^T \tilde{Z}_{\K + 1}^{\Dt}  & = H,\\
		\hat{F}^T \tilde{Z}_{k}^{\Dt}       & = - \check{F}^\top \tilde{Z}_{k + 1}^{\Dt}, \qquad k = \atob{\K}{1} .
	\end{align*}
	By induction, it is easy to see that for $k = \atob{\K + 1}{2}$, we have $\tilde{Z}_k^{\Dt} = Z_{k - 1}^{\Dt}$. Of the new adjoints, only $\tilde{Z}_1^{\Dt}$ actually has to be computed.
	
	Furthermore, note from \eqref{eq:3510} that each $\mu_{i, \ii, k}^{\Dt}$ is the integral of a product between the primal residual and the estimate of the dual error. The former only depends on $U^{\Dt}$, the latter only on $Z^{\Dt}$. When approximating the integrals using quadrature, we obtain a formula like
	\begin{equation*}
		\int_{T_{i, \ii}} \rho_{i} \lrp{U_k^{\Dt} \lrp{t}} e_{i, k}^{*^{\Dt}} \lrp{t} \tn{d} t \approx \sum_{\ii = 0}^{n_{i}} c_{\ii} \rho_{i} \lrp{U_k^{\Dt} \lrp{t_{i, \ii}}} e_{i, k}^{*^{\Dt}} \lrp{t_{i, \ii}}
	\end{equation*}
	with coefficients $c_{\ii}$. The terms $\rho_{i} \lrp{U_k^{\Dt} \lrp{t_{i, \ii}}}$ and $e_{i, k}^{*^{\Dt}} \lrp{t_{i, \ii}}$ only need to be evaluated once and from then, the quadrature coefficients can be computed by multiplying the appropriate terms.

	%--------------------------------------------   
	\subsection{Algorithm}
	\label{sec:algorithm}
	%--------------------------------------------   
	\begin{algorithm}[hbtp!]
		\caption{Goal oriented refinement for dynamical iteration.}\label{alg:GOPFEM}
		\begin{algorithmic}[1]
			\Statex \textbf{Input}: Problem \eqref{eq:1010} with $Y \in \mathcal{C}^0 \lrp{T, \mathbb{R}^m}$, $\T^0$, $\L$, $p$, $S$ and QoI \eqref{eq:1020}
			\Statex \textbf{Output}: Approximation $U_{K_{\L}}^{\Dt_{\L}}$ and error estimators
			\State Compute $L_1$, $L_2$  \qquad \hfill\Comment{see Proposition~\ref{prop:errDI}}
			\For{$l = 0, \ldots, \L$} \quad\qquad  \hfill\Comment{loop over refinement levels}
			\If{$l = 0$}
			\State Set $U_0^{\Dt_0} \equiv U_{t_0}$
			\Else
			\State $U_0^{\Dt_l} \gets U_{\K_{l - 1}}^{\Dt_{l - 1}}$
			\State Select $\left\lceil p \cdot N^{l - 1} \right\rceil$ with largest $\mu_{i, \ii}^{\Dt_{l - 1}}$ 
			\State \mbox{}\qquad Bisect corresponding cells $T_{i, \ii}^{l - 1}$ \hfill
			\Comment{refinement from $\T^{l - 1}$ to $\T^{l}$}
			\EndIf
			\State Compute $\hat{F}^l$, $G^l$, $\check{F}^l$ and $H^l$ \qquad  \hfill\Comment{Discretization, see Proposition~\ref{prop:coefficients}}
			\State Solve $\hat{F}^l U_1^{\Dt_l} = G^l - \check{F}^l U_0^{\Dt_l}$ for $U_1^{\Dt_l}$ \hfill \Comment{primal problem}
			\State Solve $\hat{F}^{l^T} Z_1^{\Dt_l} = H^l$ for $Z_1^{\Dt_l}$
			\hfill \Comment{dual problem}
			\State Set $\sup_{E_k} := \sup_{t \in T} \norm{U_1^{\Dt_l} \lrp{t} - U_0^{\Dt_l} \lrp{t}}$
			\State \Comment{Subsequent primal and dual evaluation}
			\For{$k = 2, \ldots, \K_{\max}$}
			\State Solve $\hat{F}^l U_k^{\Dt_l} = G^l - \check{F}^l U_{k - 1}^{\Dt_l}$ for $U_k^{\Dt_l}$ \hfill \Comment{primal problem}
			\For{$\tilde{k} = k, \ldots, 2$}
			\State Set $Z_{\tilde{k}}^{\Dt_l} := Z_{\tilde{k} - 1}^{\Dt_l}$
			\hfill \Comment{shift adjoint solution}
			\EndFor
			\State Solve $\hat{F}^{l^T} Z_1^{\Dt_l} = - \check{F}^{l^T} Z_2^{\Dt_l}$ for $Z_1^{\Dt_l}$ \hfill \Comment{update of dual solution}
			\State \Comment{Error Estimation}
			\For{$i = 1, \ldots, m$, $\ii = 1, \ldots, n_{i}^l$}
			\State Set $\mu_{i, \ii}^{\Dt_l} := 0$	
			\For{$k = 1, \ldots, K_l$}
			\State Compute $\mu_{k, i, \ii}^{\Dt_l}$ according to \eqref{eq:3520}
			\State Increase $\mu_{i, \ii}^{\Dt_l}$ by $\mu_{k, i, \ii}^{\Dt_l}$
			\EndFor
			\EndFor
			\State Compute total error estimator $\mu^{\Dt_l} := \sum_{i = 1}^m \sum_{\ii = 1}^{n_{i}^l} \mu_{i, \ii}^{\Dt_l}$
			\State Compute $\nu_l$ \hfill \Comment{error bound from DI}
			\If{$\mu^{\Dt_l} > \nu_l$}
			\State \textbf{break} \hfill \Comment{stop DI}
			\EndIf
			\EndFor
			\EndFor
		\end{algorithmic}
	\end{algorithm}

	Now, we have all components needed to describe the full goal oriented refinement strategy, which is presented in Algorithm~\ref{alg:GOPFEM}. Fundamentally, it consists of two nested loops:
	\begin{itemize}
		\item an outer loop over the refinement levels, $l = \atob{0}{\L}$, with predefined  maximal number of refinements $L$;
		\item an inner loop over the DI steps on each refinement level for $k = \atob{1}{\K_l}$ with $\K_l$ bounded by a given $\K_{\max}$ and a stopping criterion.
	\end{itemize}
	For the stopping criterion of the inner loop, we follow \cite{rannacher2013adaptive} and seek to balance
	\begin{equation*}
		\mu^{\Dt_l} \approx \nu_l .
	\end{equation*}
	Since additional DI steps will mostly decrease $\nu_{\K_l}$, we are specifically looking for $\geq$ in this relation.

	After an inner loop, the grid is refined by bisecting the cells with the largest local discretization error given estimators $\mu_{i, \ii}^{\Dt}$. The number of cells to be refined
	is a fixed predefined fraction $p \in \left( 0, 1 \right]$ of the total number of cells. The case $p = 1$ corresponds to uniform refinement.
	
	If $l < \L$ (outer loop), the computed refinement is used in the next series of dynamic iterations. The final approximation $U_{\K_l}^{\Dt_l}$ on the current mesh will serve as the initial waveform $U_{0}^{\Dt_{l + 1}}$ which allows us to keep the progress of the DI. 
	
	\begin{remark}
		(i) The error estimators $\mu_{i, \ii}^{\Dt_l}$ which were used to decide which cells need to be refined were computed with old iterates, while mesh is used for the new iterates.
		This resembles the classical heuristics for time stepping.
		
		(ii) The discretization error estimator is computed for the DI process \emph{up to the last refinement}. One practical consequence for splitting is that enough iteration steps need to be performed to allow information to propagate through the system. For instance if only a single step with Jacobi splitting was performed, no coupling between components has taken place yet. Consequently a component which does not directly impact the quantity of interest will not get any refinement.
	\end{remark}
	
	%----------------------------------------
	\section{Numerical Experiments}
	\label{sec:experiments}
	%----------------------------------------
	We present the numerical results for the goal oriented refinement for DI, Algorithm~\ref{alg:GOPFEM}. We compare the explicit Euler scheme and the Crank--Nicolson scheme, each with goal oriented and uniform refinement. 
	Some parameters will remain common between all following test cases. We use $\L = 10$ goal oriented refinement steps with $p = 0.4$ and $\L = 5$ uniform ($p = 1$) steps yielding roughly similar numbers of total time steps, starting from an initial grid consisting of $32$ equidistant time points for each component.
	
	%---------------------------------
	\begin{experiment} % TEST 1
		\label{exp:6110}
		%---------------------------------
		On $\timeinterval := \lrb{0, 3}$, we consider the weakly coupled initial value problem
		\begin{equation*}
			\begin{pmatrix} \dot{u}_1 \lrp{t}\\ \dot{u}_2 \lrp{t} \end{pmatrix} + \lrb{\begin{array}{cc}  10 & -1\\  1 &  10 \end{array}} \begin{pmatrix} u_1 \lrp{t}\\ u_2 \lrp{t} \end{pmatrix} = \begin{pmatrix} 10 \sin \lrp{t}\\ \sin \lrp{10 t} \end{pmatrix},
			\qquad 
			\begin{pmatrix}
				u_1 \lrp{0} \\
				u_2 \lrp{0}  
			\end{pmatrix}
			=\left(\begin{array}{r}
				-0.1 \\ 
				0.1
			\end{array}\right)
		\end{equation*}
		and the quantity of interest (QoI) reads
		\begin{equation*}
			J \lrp{U} := u_1 \lrp{2} + u_1 \lrp{3} + 2 u_2 \lrp{3} .
		\end{equation*}
		Figure~\ref{fig:6110} summarizes the results obtained with a Jacobi--splitting. 
		We can see that the goal oriented refinement generally achieves better result than the uniform. 
		As expected Crank--Nicolson scheme is more accurate than the Euler scheme due to its higher order.
		Moreover, the Crank--Nicolson method gives a more consistent error reduction as the grid is refined. 
		On the other hand, the error estimators fall below the actual errors for Crank--Nicolson; this is not the case for the explicit Euler method.
	\end{experiment}

	\begin{figure}[htb]
		\centering
		\includegraphics[width = \textwidth]{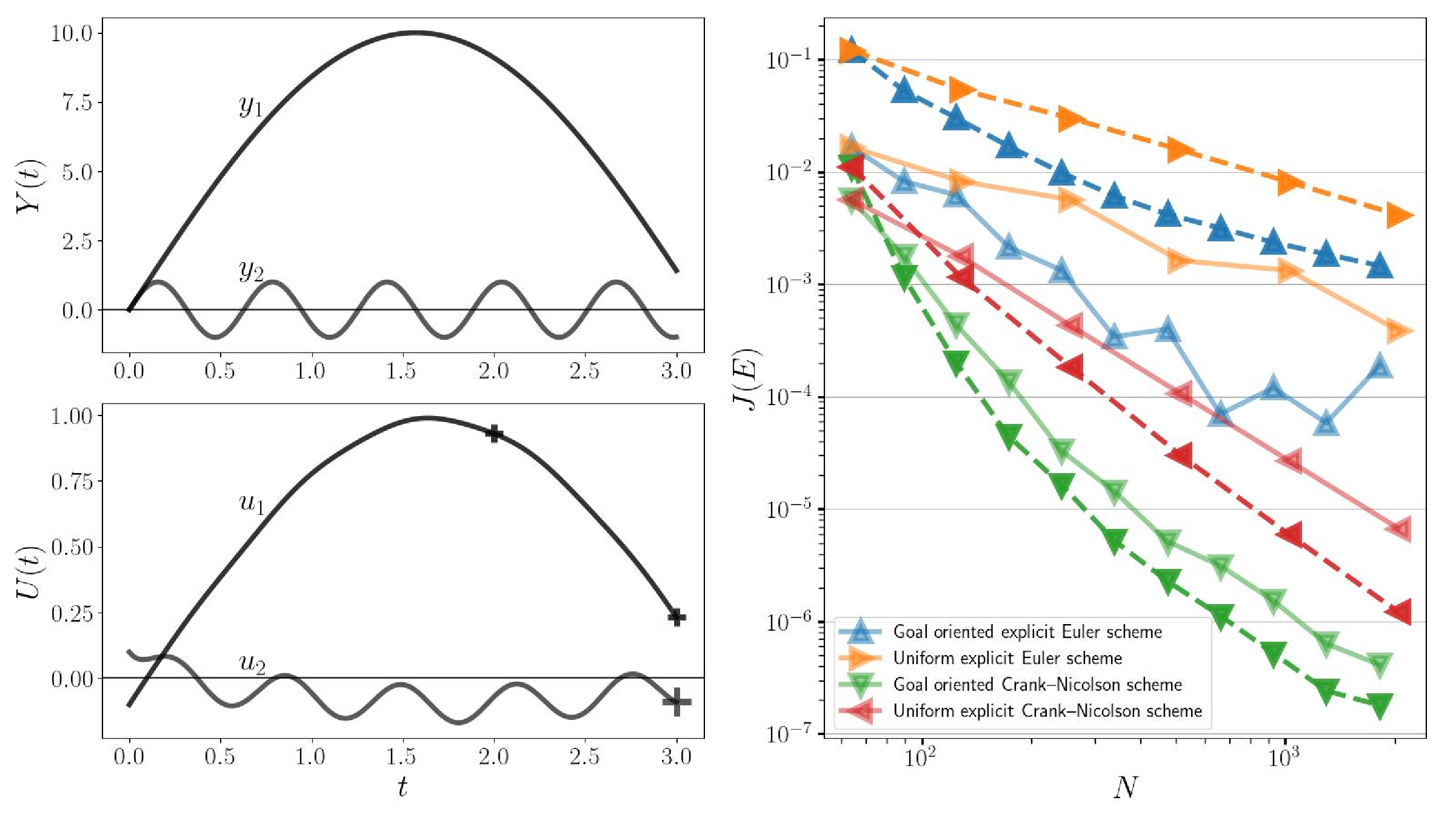}
		\caption{Test Case \ref{exp:6110}. Top left: right-hand side. Bottom left: analytical solution. The ''+'' markers denote the time points of interest which we try to approximate. The larger marker for $u_2$ at $t = 3$ represents the double weight there. Right: Error plots. The $x$--axis displays the total number of time steps made over all components; the $y$--axis shows the error (solid) in $J$ or the estimator (dashed). 
			The goal oriented explicit Euler scheme (blue, up marker) performs better than its uniform counterpart (orange, right marker) in both cases, the error is bounded from above by the estimator. Lower errors can be achieved by using the higher order Crank--Nicolson scheme; with uniform refinement (red, left marker), this yields a very consistent error reduction as $N$ grows while goal oriented refinement (green, down marker) leads to the best results of the examined methods. However, for the Crank--Nicolson schemes, the error is underestimated.
		}
		\label{fig:6110}
	\end{figure}
	
	%-------------------------
	\begin{experiment} % Test 2
		\label{exp:6210}
		%--------------------------
		For $\timeinterval := \lrb{0, 2.5}$, we consider the initial value problem
		\begin{equation*}
			\begin{pmatrix} \dot{u}_1 \lrp{t}\\ \dot{u}_2 \lrp{t}\\ \dot{u}_3 \lrp{t}\\ \dot{u}_4 \lrp{t} \end{pmatrix} + \lrb{\begin{array}{cccc}  5 &  0 &  0 &  0\\  2 &  5 &  1 &  0\\  2 &  0 &  5 &  1\\  0 &  0 & -1 &  5 \end{array}} \begin{pmatrix} u_1 \lrp{t}\\ u_2 \lrp{t}\\ u_3 \lrp{t}\\ u_4 \lrp{t} \end{pmatrix} = \begin{pmatrix} 10 \sin \lrp{t}\\ - 10 \sin \lrp{t}\\ \sin \lrp{10 t}\\ - \sin \lrp{t} \end{pmatrix}
			,\qquad
			\begin{pmatrix}
				u_1 \lrp{0}\\
				u_2 \lrp{0} \\
				u_3 \lrp{0}\\ 
				u_4 \lrp{0}
			\end{pmatrix}
			= \left(\begin{array}{r}
				-0.4\\
				-0.2\\
				0.2\\
				0.4 
			\end{array}\right)
		\end{equation*}
		with QoI, which considers two time points far away from one another in two components:
		\begin{equation*}
			J \lrp{U} := u_2 \lrp{0.5} + u_3 \lrp{2.5} .
		\end{equation*}
		The analytic solution is shown in Figure~\ref{fig:6210} (lower left).
		This problem is split between into slow components $u_1$ and $u_2$ and the fast components $u_3$ and $u_4$. Notice the QoI picks a slow an a fast component. 
	\end{experiment}
	The dynamic iteration and refinement results in Figure~\ref{fig:6210} indicate that the goal oriented refinement performs very well. Again Crank--Nicolson in goal-oriented form outperforms the other schemes. However, the Crank--Nicolson scheme still gives a slightly unreliable bound of the total error. 
	
	\begin{figure}[htb]
		\centering
		\includegraphics[width = \textwidth]{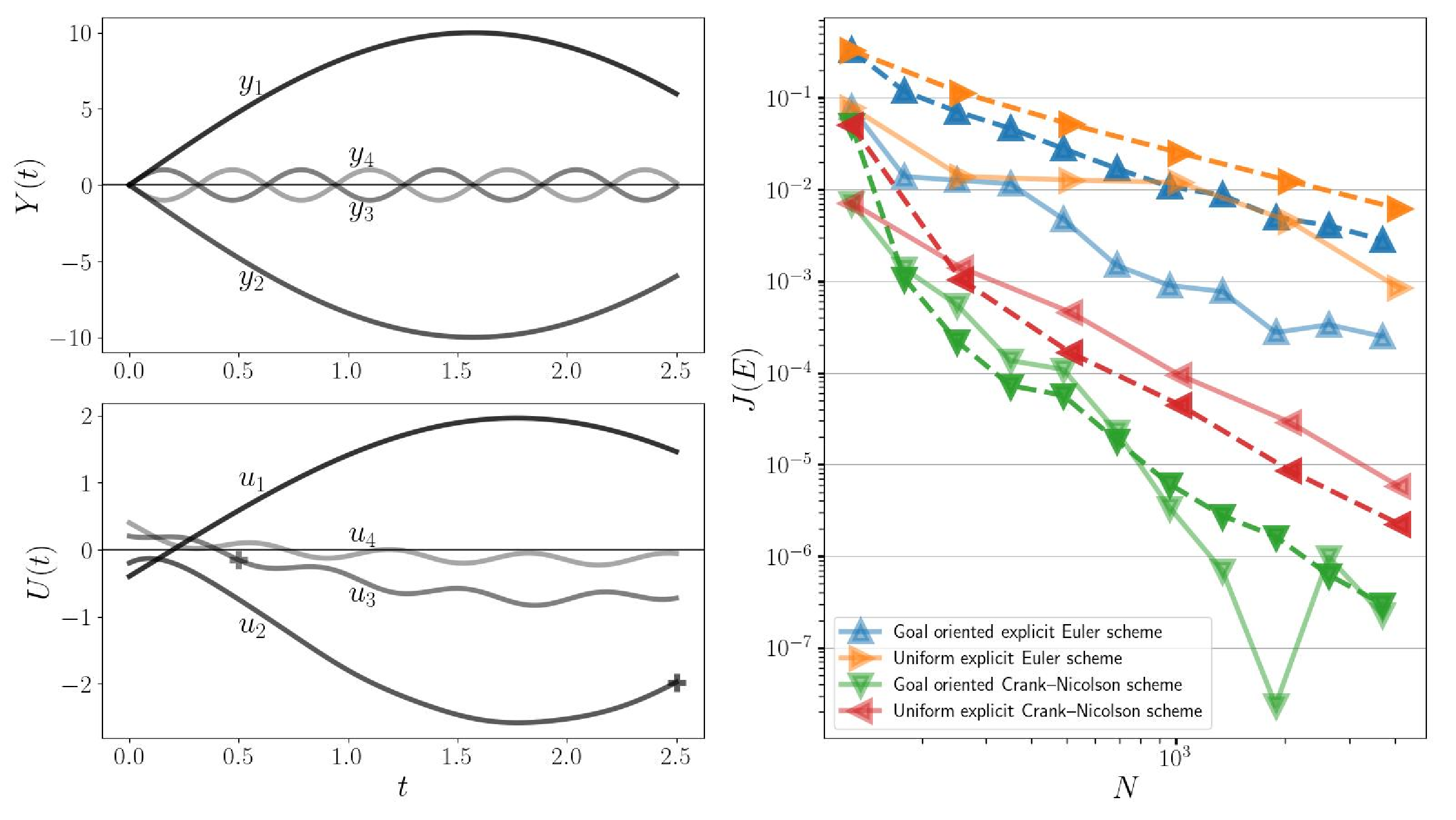}
		\caption{Test Case \ref{exp:6210}. Top left: right--hand side with slow components $y_1$, $y_2$ and faster components $y_3$, $y_4$. Bottom left: analytical solution with slow components $u_1$, $u_2$ and faster components $u_3$, $u_4$. The ''+'' also denote the points of interest; We measure $u_2$ at $t = 0.5$ and at $t = 2.5$, we take $u_3$.
			Right: Convergence results. For the explicit Euler scheme, we see that the goal oriented refinement (blue, down marker) outperforms uniform (orange, right marker) which does not work well in this case. The uniform Crank--Nicolson scheme (red, left marker) reliably reduces the error but also underestimates it (dashed line below solid line). Finally, the goal oriented Crank--Nicolson scheme (green, down marker) works very well here although the error does not monotonically decrease as $l$ is incremented. While the estimator does at times fall below the actual error, the effect is less prominent in this test case due to the high accuracy of goal oriented Crank--Nicolson.}
		
		\label{fig:6210}
	\end{figure}
	
	To investigate the behavior in the refinement, the results on the final mesh are shown in Figure \ref{fig:6230}.
	We observe for the goal-oriented dynamic iteration with Euler and Crank--Nicolson:
	\begin{itemize}
		\item $u_1$ is refined mostly just before $t = 0.5$ and $t = 0.25$, where $u_2$ and $u_3$ enter the QoI. Note, $u_1$ is needed to compute these quantities.
		\item $u_2$ is refined primarily toward the end where it is evaluated for the goal functional. Since it does not influence any other component, there is no need to refine it elsewhere.
		\item $u_3$ sees refinement mostly before $t = 0.5$, where we need it for $J$ and near the end $u_2$ depends on it.
		\item $u_4$ is mostly refined before $t = 0.5$ since it is coupled with $u_3$. Some refinement also takes place after that since $u_4$ and $u_2$ are indirectly coupled via $u_3$ though this effect is weaker resulting in less refinement than e.g. $u_1$ or $u_3$.
	\end{itemize}
	A notable difference is that the Crank--Nicolson scheme results in a more distributed refinement while the explicit Euler scheme results in a particularly strong concentration on several subintervals of $\timeinterval$. %\hfill $\Box$
	
	\begin{figure}[htb]
		\centering
		\includegraphics[width = \textwidth]{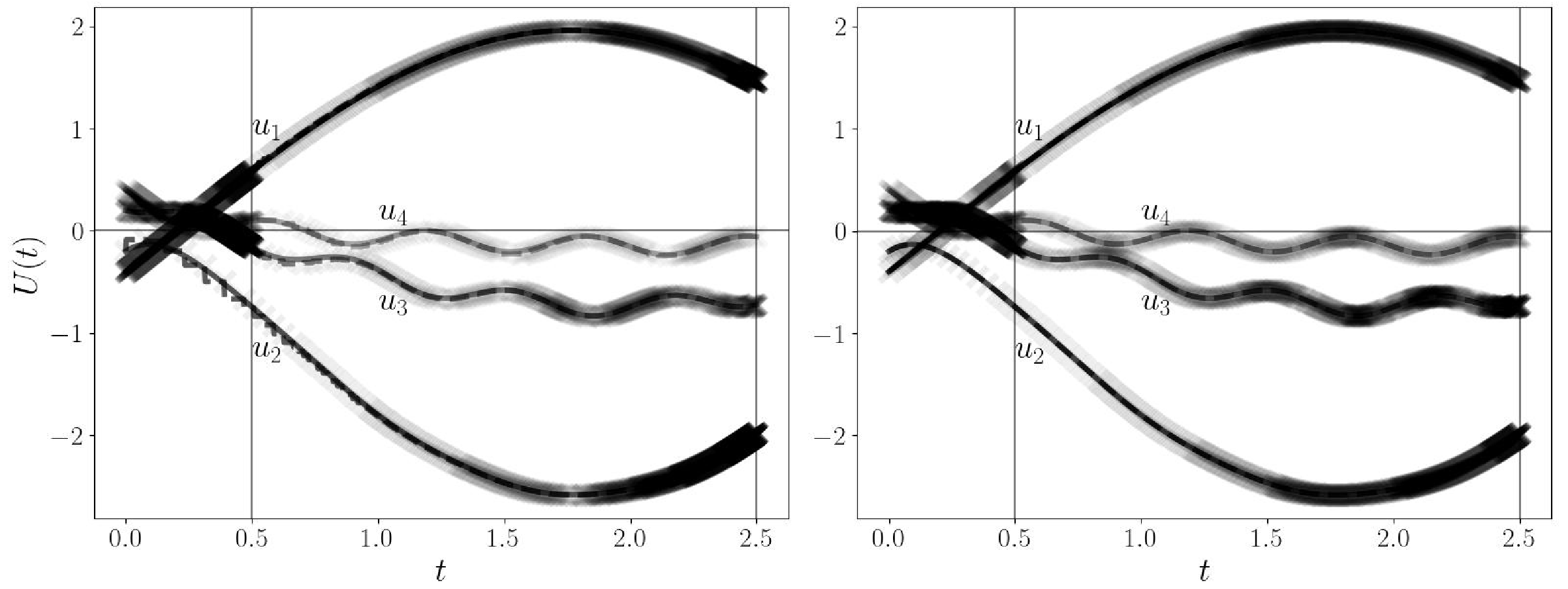}
		\caption{Test Case \ref{exp:6210} and final approximations for goal oriented grid refinement. The time steps are marked with an $X$ each. Left: explicit Euler scheme. Right: Crank--Nicolson scheme. The refinement adapts to the goal functional, here represented by a vertical line, as described above. It is also notable that the Crank--Nicolson scheme while having the same pattern of refinement as the explicit Euler scheme distributes the time steps slightly more evenly.}
		\label{fig:6230}
	\end{figure}
	
	%---------------------------
	\begin{experiment}  % TEST 3
		\label{exp:6310}
		%---------------------------
		On $\timeinterval := \lrb{0, 4}$, we investigate the initial value problem
		\begin{equation*}
			\begin{pmatrix} \dot{u}_1 \lrp{t}\\ \dot{u}_2 \lrp{t} \end{pmatrix} + \lrb{\begin{array}{cc}  5 & 2\\  1 &  2.5 \end{array}} \begin{pmatrix} u_1 \lrp{t}\\ u_2 \lrp{t} \end{pmatrix} = \begin{pmatrix} 10 \sin \lrp{t} + 0.1 \sin \lrp{10 t}\\ \sin \lrp{t} +\ \sin \lrp{10 t} \end{pmatrix},
			\qquad
			\begin{pmatrix}
				u_1 \lrp{0} \\ u_2 \lrp{0}
			\end{pmatrix}
			= \left( \begin{array}{r} -0.5 \\ 0.5\end{array} \right)
		\end{equation*}
		and QoI
		\begin{equation*}
			J \lrp{U} := u_1 \lrp{3} + u_2 \lrp{4} .
		\end{equation*}
		
		The results with Jacobi--splitting are shown and shortly summarized in Figure~\ref{fig:6310}. Additionally, we measure the number of dynamic iteration steps per refinement level, see Figure~\ref{fig:6330}. On the initial level, the maximum number of iterations is performed before dropping sharply. For the explicit Euler scheme, $\K_l$ generally grows with $l$ while the Crank--Nicolson scheme shows much less regularity. %\hfill $\Box$
		The reason later refinement levels require more dynamic iteration steps is because the discretization error is lower there. From a computational cost point of view it would be preferable to have a high number of iterations at the start but reduce them as the grid gets finer (and thus each iteration step gets more expensive).
		
		\begin{figure}[htb]
			\centering
			\includegraphics[width = \textwidth]{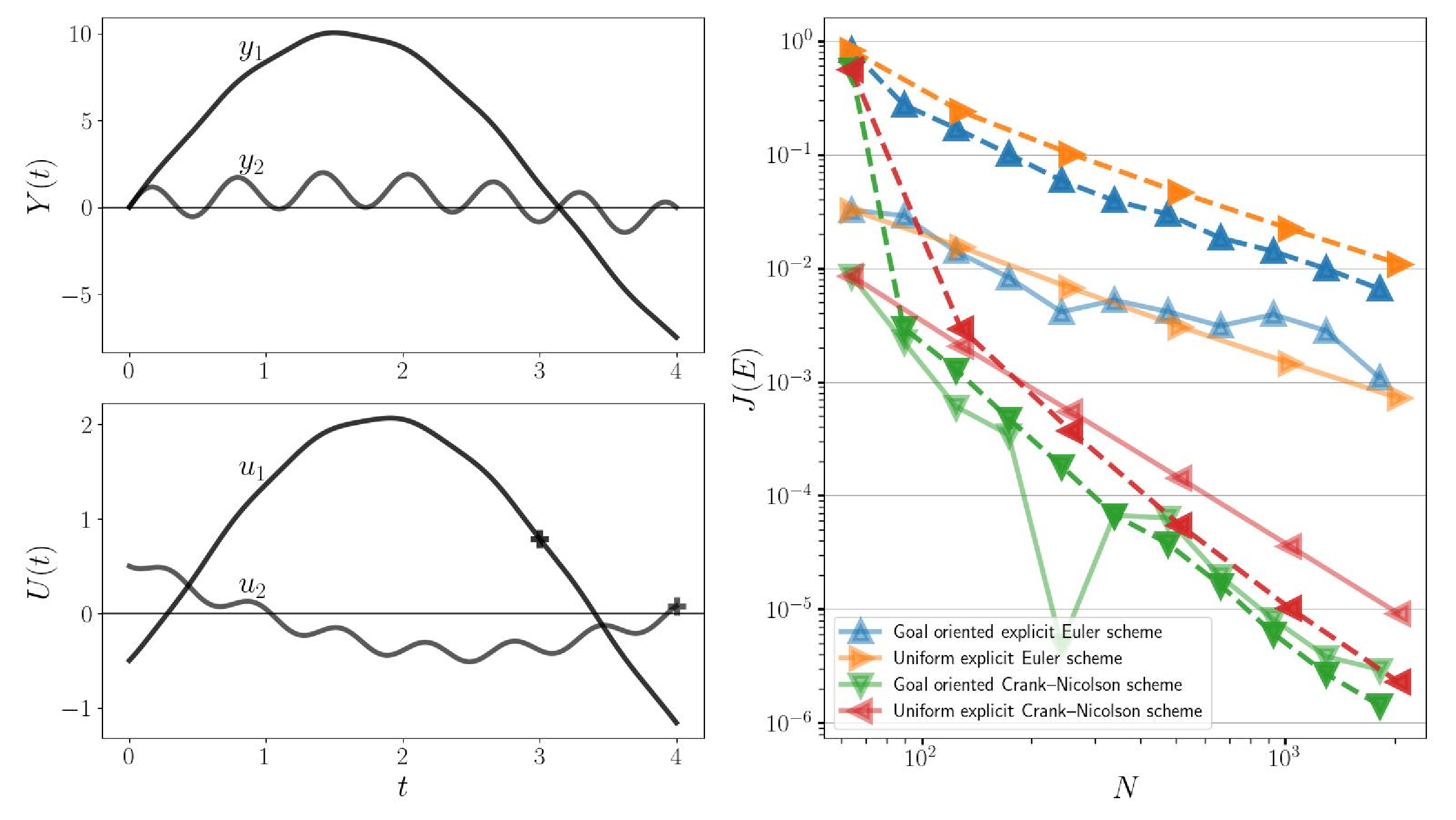}
			\caption{Test Case \ref{exp:6310}. Top left: right-hand side. Bottom left: analytical solution with points of interest marked by ''+''. Right: convergence plots for the goal oriented dynamic iteration refinements. The goal oriented approach offers less of an advantage here than in the previous tests. In particular, for the explicit Euler scheme, the goal oriented (blue, up marker) refinement performs worse than the uniform (orange, right marker) one. For the Crank--Nicolson scheme, goal oriented (green, down marker) still works better than uniform (red, left marker) but the difference is smaller than in other cases. The stronger coupling compared to Test Case \ref{exp:6110} makes it more difficult to exploit the different frequencies and points of interest of the components.}
			\label{fig:6310}
		\end{figure}
		
		\begin{figure}[htb]
			\centering
			\includegraphics[width = \textwidth]{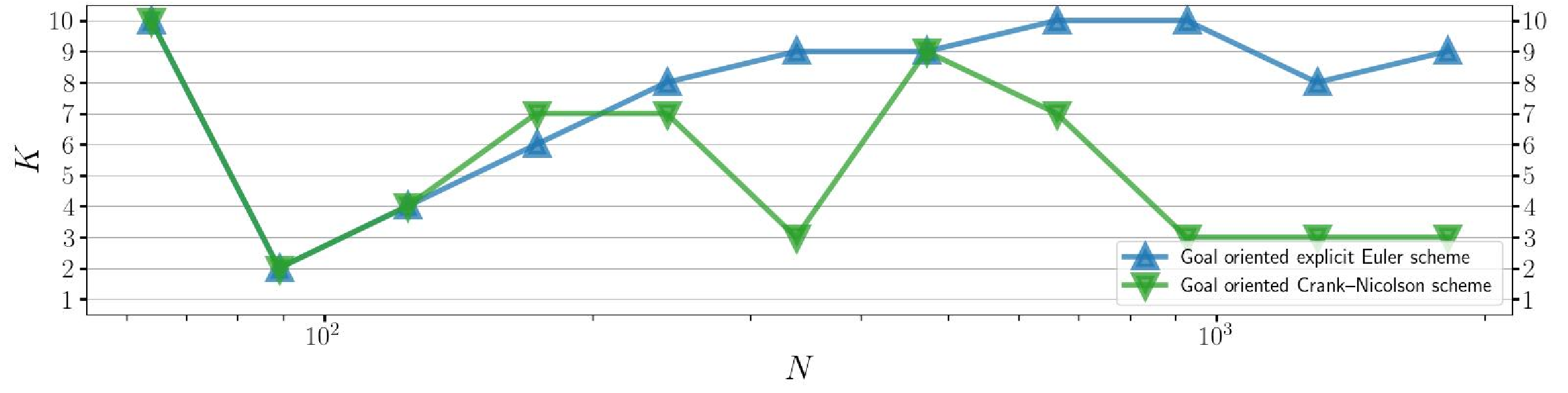}
			\caption{Test Case \ref{exp:6310}. Iteration steps ($K$) per refinement level for the goal oriented schemes. Explicit Euler (blue, up marker) scheme increases the number of iterations for finer meshes while the Crank--Nicolson scheme (green, down marker) does not shown a clear pattern.}
			\label{fig:6330}
		\end{figure}
	\end{experiment}
	
	\section{Conclusion}
	\label{sec:conclusion}
	In this paper, we have combined results from DI and DWR theories to construct an algorithm for the adaptive solution of time dependent systems of ODEs. Goal oriented error estimators are derived for both the splitting and discretization errors and these are used to control the number of dynamic iteration steps as well as to guide the refinement of the mesh.
	
	We have adopted a view of the entire DI process as a single system. This leads to a rigorous derivation of the adjoint problem where information flows in inverse direction compared to the primal problem. In particular, this means that the iterations of the dual problem have to be computed in reverse order. However, it is possible to significantly reduce the cost of this by reusing certain results.
	
	These first experiments show that the resulting scheme is advantageous compared to a naive uniform refinement. This applies especially to cases where there are large differences in the frequencies of the components and the time points of interest. Investigations of the mesh refinement pursued also reveal results that are intuitively sensible. The error estimators were reliable for the explicit Euler scheme but not fully reliable for the higher order Crank--Nicolson scheme. Nonetheless, they provide useful information for the purposes of mesh refinement.
	
	There are various possible avenues for further research. Besides higher order methods and different splitting methods, 
	a question of particular interest is whether any information gained by the goal oriented error estimation could be used to determine an optimal splitting.

	The discrete goal functional used in the present paper can also be an ideal use case for a windowing technique. In that case, the goal oriented error estimators could be used to find windows that contribute most to the overall error. Moreover, distributed QoIs are still missing in our investigations. 
	Finally, the method needs to be developed for a wider category of problems including PDEs or nonlinear problems. 
	
	\backmatter
	
	% \clearpage
	
	\section*{Declarations}
	\textbf{Conflict of Interest} The authors declare no conflict of interest.
	
	\noindent
	\textbf{Funding} The second and third author acknowledges support by the Deutsche Forschungsgemeinschaft (DFG, German Research Foundation) Project-ID 531152215 – CRC 1701.
	
	\noindent
	\textbf{Code Availability} The program code for the Algorithm as well as the scripts to produce the graphics presented is provided on \href{https://git.uni-wuppertal.de/eweyl/goal-oriented-partitioning}{GitLab}.
	
	\begin{appendices}
		
		\section{Proof of DWR Error Estimator}
		\label{sec:appendix050}	
		The proof is classical in goal oriented techniques and follows the arguments of the proof of \cite[Proposition 2.2]{becker2001optimal}. We define the primal and dual residuals for $\Kblock{V} \in \Kblock{\mathcal{V}}$, $\Kblock{W} \in \Kblock{\mathcal{U}}$:
		\begin{align*}
			\rho \lrp{\KblockDt{U}, \Kblock{V}}		& := \Kblock{g} \lrp{\Kblock{V}} - \Kblock{f} \lrp{\KblockDt{U}, \Kblock{V}},\\
			\rho^* \lrp{\Kblock{W}, \KblockDt{Z}}	& := \Kblock{J} \lrp{\Kblock{W}} - \Kblock{f} \lrp{\Kblock{W}, \KblockDt{Z}}.
		\end{align*}
		Note that the residuals vanish for any $\Kblock{V}\in \KblockDt{\mathcal{V}}$, $\Kblock{W} \in \KblockDt{\mathcal{U}}$.
		Using $\Kblock{U} \in \KblockDt{\mathcal{U}}$ and $\Kblock{Z} \in \KblockDt{\mathcal{V}}$, we obtain for the difference in the Lagrangian:
		\begin{align*}
			L \lrp{\Kblock{U}, \Kblock{Z}} - & L \lrp{\KblockDt{U}, \KblockDt{Z}}
			%& - \Kblock{J} \lrp{\KblockDt{U}} - \underbrace{\Kblock{g} \lrp{\KblockDt{Z}} + \Kblock{f} \lrp{\KblockDt{U}, \KblockDt{Z}}}_{= 0}\\
			= \Kblock{J} \lrp{\Kblock{U}} - \Kblock{J} \lrp{\KblockDt{U}} 
		\end{align*}
		and we have expressed the error in $\Kblock{J}$ as an error in the Lagrangian functional. Consequently, we have in our linear setting:
		\begin{align*}
			L &\lrp{\Kblock{U}, \Kblock{Z}} - L \lrp{\KblockDt{U}, \KblockDt{Z}}\\
			=	& \quad \ \frac{1}{2} \min_{\KblockDt{W} \in \KblockDt{\mathcal{U}}} \lrp{\pardif[L]{\Kblock{U}} \lrp{\KblockDt{U}, \KblockDt{Z}} \lrb{\Kblock{U} - \KblockDt{W}}}\\
			& + \frac{1}{2} \min_{\KblockDt{V} \in \KblockDt{\mathcal{V}}} \lrp{\pardif[L]{\Kblock{Z}} \lrp{\KblockDt{U}, \KblockDt{Z}} \lrb{\Kblock{Z} - \KblockDt{V}}} \\
			=	& \quad \ \frac{1}{2} \min_{\KblockDt{W} \in \KblockDt{\mathcal{U}}} \lrp{\rho^* \lrp{\Kblock{U} - \KblockDt{W}, \KblockDt{Z}}}\\
			& + \frac{1}{2} \min_{\KblockDt{V} \in \KblockDt{\mathcal{V}}} \lrp{\rho \lrp{\KblockDt{U}, \Kblock{Z} - \KblockDt{V}}} .
		\end{align*}
		Thus, for arbitrary $\KblockDt{V} \in \KblockDt{\mathcal{U}}$,
		\begin{align*}
			\Kblock{J} & \lrp{\Kblock{U}} - \Kblock{J} \lrp{\KblockDt{U}}\\
			\leq	& \quad \ \rho \lrp{\KblockDt{U}, \Kblock{Z} - \KblockDt{V}}\\
			=		& \quad \ \Kblock{g} \lrp{\Kblock{Z} - \KblockDt{V}} - \Kblock{f} \lrp{\KblockDt{U}, \Kblock{Z} - \KblockDt{V}}\\
			=		& {\int_\timeinterval \!\!\lrp{\Kblock{Y} \lrp{t} \!-\! \KblockDt{\dot{U}} \lrp{t} \!-\! \Kblock{B} \KblockDt{U} \lrp{t}} \cdot \lrp{\Kblock{Z} \lrp{t} - \KblockDt{V} \lrp{t}} \tn{d} t} ={\mu}^{\Dt} .
		\end{align*}
		This completes the proof.
	\end{appendices}

\end{document}